\documentclass[a4paper]{amsart}
\usepackage{amsthm,amsfonts,amsmath,amssymb}
\usepackage[abs]{overpic}
\usepackage{comment} 
\usepackage{color} 
\usepackage{bm}

\newtheorem{theorem}{Theorem}[section]
\newtheorem{lemma}[theorem]{Lemma}

\theoremstyle{definition}

\newtheorem{remark}[theorem]{Remark}

\newtheorem{example}[theorem]{Example}
\newtheorem*{acknowledgements}{Acknowledgements}

\makeatletter

\@addtoreset{figure}{section}
\makeatother

\begin{document}

\title{Link invariants derived from multiplexing of crossings}
\author{Haruko A. Miyazawa, Kodai Wada and Akira Yasuhara}

%1
\address{
Institute for Mathematics and Computer Science, Tsuda University,
2-1-1 Tsuda-Machi, Kodaira, Tokyo, 187-8577, Japan}
\email{aida@tsuda.ac.jp}

%2
\address{Department of Mathematics, Graduate School of Education, Waseda University, 1-6-1 Nishi-Waseda, Shinjuku-ku, Tokyo, 169-8050, Japan}
\email{k.wada@akane.waseda.jp}

%3
\address{
Department of Mathematics, Tsuda University, 
2-1-1 Tsuda-Machi, Kodaira, Tokyo, 187-8577, Japan}
\email{yasuhara@tsuda.ac.jp}

% mathesubject
% keywords
% grant 
%\thanks{This work was supported by JSPS KAKENHI Grant Numbers JP26400098, JP17J08186, JP17K05264.}

%\date{\today}

%%%%%%%%%% abstract %%%%%%%%%%
\begin{abstract}
We introduce the multiplexing of a crossing, 
replacing a classical crossing of a virtual link diagram 
with multiple crossings which is a mixture of classical and virtual.
For integers $m_{i}$ $(i=1,\ldots,n)$ and an ordered 
$n$-component virtual link diagram $D$,  
a new virtual link diagram $D(m_{1},\ldots,m_{n})$ is obtained 
from $D$ by the multiplexing of all crossings. 
For welded isotopic virtual link diagrams $D$ and $D'$, 
$D(m_{1},\ldots,m_{n})$ and $D'(m_{1},\ldots,m_{n})$ 
are welded isotopic. 
From the point of view of classical link theory, it seems very 
interesting that $D(m_{1},\ldots,m_{n})$ could not be welded isotopic to a classical link diagram 
even if $D$ is a classical one, and 
new classical link invariants are expected from known welded link invariants via 
the multiplexing of crossings. 
\end{abstract}

\maketitle

%%%%%%%%%% Introduction %%%%%%%%%%
\section{Introduction}
An $n$-component {\em virtual link diagram} is 
a generic immersed $n$ circles in a plane 
whose singularities are transverse double points, 
that are labeled either as a {\it classical crossing} or as a {\it virtual crossing} as illustrated in Figure~\ref{xing}. 
Note that we do not use here the usual drawing convention for virtual crossings. 
{\em Virtual isotopy} is an equivalence relation on virtual link diagrams generated by classical Reidemeister moves R1--3 
and virtual Reidemeister moves VR1--4 as illustrated in Figure~\ref{GRmoves}.
We remark that VR1--4 imply a {\em detour move}, 
which replaces an arc passing through a number of virtual crossings with any other such arc, with same endpoints. 
{\em Welded isotopy} is the extension of virtual isotopy which also allows the move OC as illustrated in Figure~\ref{OC}.
(Note that OC stands for Overcrossings Commute.)
A {\em welded link} is an equivalence class of virtual link diagrams 
under welded isotopy.
M. Goussarov, M. Polyak and O. Viro~\cite{GPV} essentially proved that 
welded isotopic classical link diagrams are equivalent, 
that is, they are deformed into each other by classical Reidemeister moves.  
Therefore, we can consider welded links as a natural generalization of the classical links.

In this paper, we introduce the {\em multiplexing} of a crossing for a virtual link diagram, 
as a local change on a classical crossing shown in Figure~\ref{multiplexing}.
Let $m_{i}$ be integers $(i=1,\ldots,n)$ and $D$ an ordered $n$-component virtual link diagram. 
By the multiplexing of all classical crossings of $D$, 
we obtain the virtual link diagram $D(m_{1},\ldots,m_{n})$ 
from $D$ associated with $(m_{1},\ldots,m_{n})$, 
see Section~\ref{sec-multiplexing} for the precise definition. 
We show that if virtual link diagrams $D$ and $D'$ are welded isotopic, 
then $D(m_{1},\ldots,m_{n})$ and $D'(m_{1},\ldots,m_{n})$ 
are welded isotopic for any $(m_{1},\ldots,m_{n})$ $\in\mathbb{Z}^{n}$ (Theorem~\ref{th-multiplexing}).

\begin{figure}[t]
  \begin{center}
    \begin{overpic}[width=5cm]{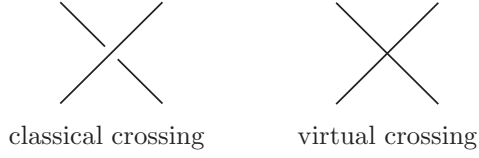}
      \put(-19,-15){classical crossing}
      \put(89,-15){virtual crossing}
    \end{overpic}
  \end{center}
  \vspace{1em}
  \caption{Classical and virtual crossings}
  \label{xing}
\end{figure}

\begin{figure}[t]
  \begin{center}
    \begin{overpic}[width=12cm]{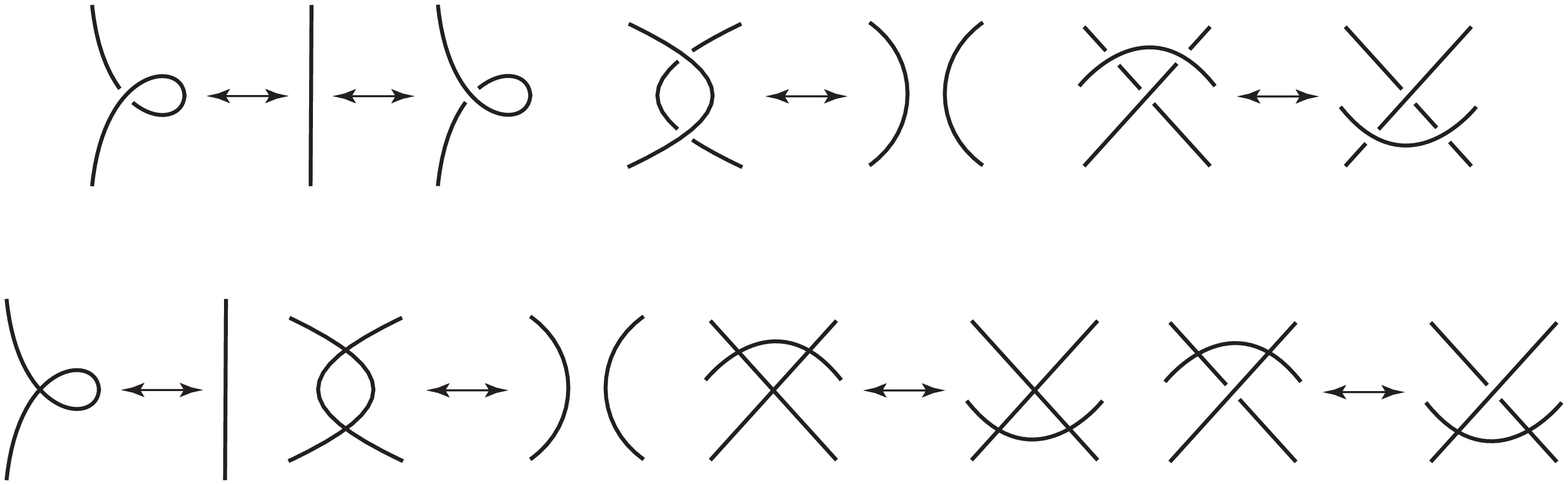}
      \put(24.5,25){VR1}
      \put(90.5,25){VR2} 
      \put(186,25){VR3}
      \put(287,25){VR4}
      \put(47.5,88){R1}
      \put(75,88){R1}
      \put(169.2,88){R2}
      \put(272.5,88){R3}
    \end{overpic}
  \end{center}
  \caption{Classical and virtual Reidemeister moves}
  \label{GRmoves}
\end{figure}

\begin{figure}[t]
  \begin{center}
    \begin{overpic}[width=4cm]{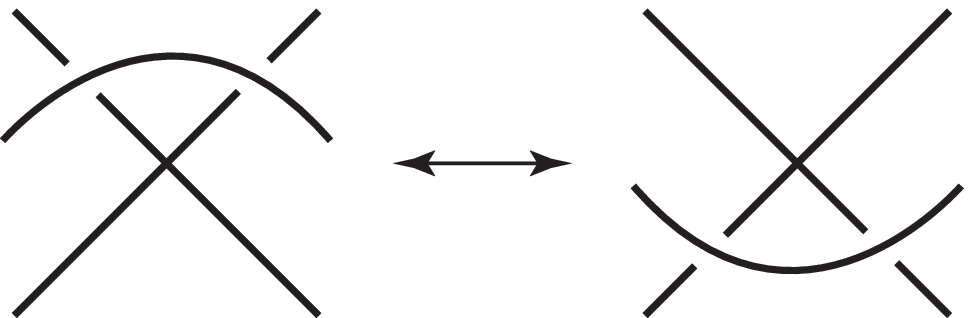}
      \put(50,23){OC}
    \end{overpic}
  \end{center}
  \caption{Move OC}
  \label{OC}
\end{figure}

The {\em group} of a virtual link diagram is known as a welded link invariant~\cite{K}.  
Hence by Theorem~\ref{th-multiplexing}, we have that the group $G(D(m_{1},\ldots,m_{n}))$ of 
$D(m_{1},\ldots,m_{n})$ is a link invariant of $D$. 
We remark that $G(D(m,\ldots,m))$ is isomorphic to 
the {\em generalized link group $G_{m}(D)$} defined by 
A.J. Kelly~\cite{Kelly} and M. Wada~\cite{W}, independently. 
Therefore, 
$G(D(m_{1},\ldots,m_{n}))$ is a generalization of $G_{m}(D)$. 
As an application, we show that 
for a non-zero integer $m$ and 
for {\em classical knot} diagrams $D$ and $D'$, 
$D$ is equivalent to $D'$ or its mirror image if and only if $D(m)$ is welded isotopic to $D'(m)$ or its mirror image (Theorem~\ref{th-iff}).

From the point of view of classical link theory, it seems very 
interesting  that $D(m_{1},\ldots,m_{n})$ could not be welded isotopic to a classical link diagram 
even if $D$ is a classical one, and 
new classical link invariants are expected from known welded link invariants via 
the multiplexing of crossings. 
For example, there is a $3$-component {\em classical} link diagram $D$ with trivial Alexander polynomial such that 
for $m_1\neq m_2$ and $m_3\neq 0$, the Alexander polynomial of $D(m_{1},m_{2},m_{3})$ is non-trivial  
and that $D(m_{1},m_{2},m_{3})$ is not welded isotopic to a classical link diagram~(Example~\ref{ex-newinv}).

%%%%%%%%%% Multiplexing of crossings %%%%%%%%%%
\section{Multiplexing of crossings}
\label{sec-multiplexing}
Let $(m_{1},\ldots,m_{n})$ be an ordered set of integers and $D=D_{1}\cup\cdots\cup D_{n}$ an ordered $n$-component virtual link diagram.  
For a classical crossing of $D$ whose overpass belongs to $D_{j}$, 
we define the {\em multiplexing} of the crossing associated with $m_{j}$ as a local change shown in Figure~\ref{multiplexing}. 
When $m_{j}=0$, the multiplexing of the crossing is the virtualization of it. 
The number of classical crossings that
appear in the multiplexing of the crossing is the absolute value of $m_{j}$.
Let $D(m_{1},\ldots,m_{n})$ denote the virtual link diagram 
obtained from $D$ by the multiplexing of all classical crossings of $D$ associated with $(m_{1},\ldots,m_{n})$. 
Then we have the following theorem.

\begin{figure}[htbp]
  \begin{center}
    \begin{overpic}[width=10cm]{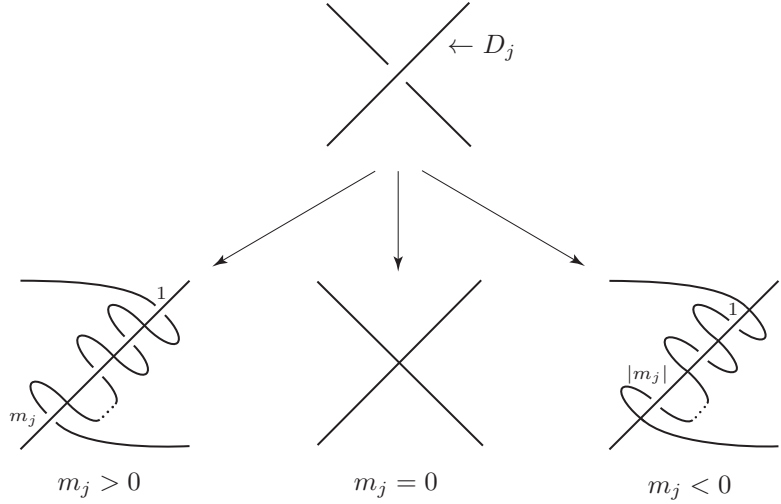}
      \put(14,-15){$m_{j}>0$}
      \put(125,-15){$m_{j}=0$}
      \put(235,-16){$m_{j}<0$}
      \put(160,150){$\leftarrow$ $D_{j}$}
      \put(51,57){{\scriptsize $1$}}
      \put(-4,11){{\scriptsize $m_{j}$}}
      \put(265,51){{\scriptsize $1$}}
      \put(227,27){{\scriptsize $|m_{j}|$}}
    \end{overpic}
  \end{center}
  \vspace{1em}
  \caption{Multiplexing of a crossing}
  \label{multiplexing}
\end{figure}

\begin{theorem}
\label{th-multiplexing}
If ordered $n$-component virtual link diagrams $D$ and $D'$ are welded isotopic, 
then for any $(m_{1},\ldots,m_{n})\in\mathbb{Z}^{n}$, $D(m_{1},\ldots,m_{n})$ and $D'(m_{1},\ldots,m_{n})$ 
are welded isotopic.
\end{theorem}

\begin{remark}
There are equivalent classical link diagrams $D$ and $D'$ such that 
$D(m_{1},\ldots,m_{n})$ and $D'(m_{1},\ldots,m_{n})$ are not {\em virtual} isotopic for some $(m_{1},\ldots,m_{n})$. 
For example, let $D$ be the classical knot diagram illustrated in the left-hand side of Figure~\ref{rem-virtual}. 
Then the virtual knot diagram $D(2)$ is not virtual isotopic to the 
trivial one~\cite{K}. 
Let $D'$ be the trivial knot diagram without crossings, 
then $D'(2)=D'$. 
Therefore, $D$ and $D'$ are equivalent, but $D(2)$ and $D'(2)$ are not virtual isotopic.
\end{remark}

\begin{figure}[htbp]
  \begin{center}
    \begin{overpic}[width=6cm]{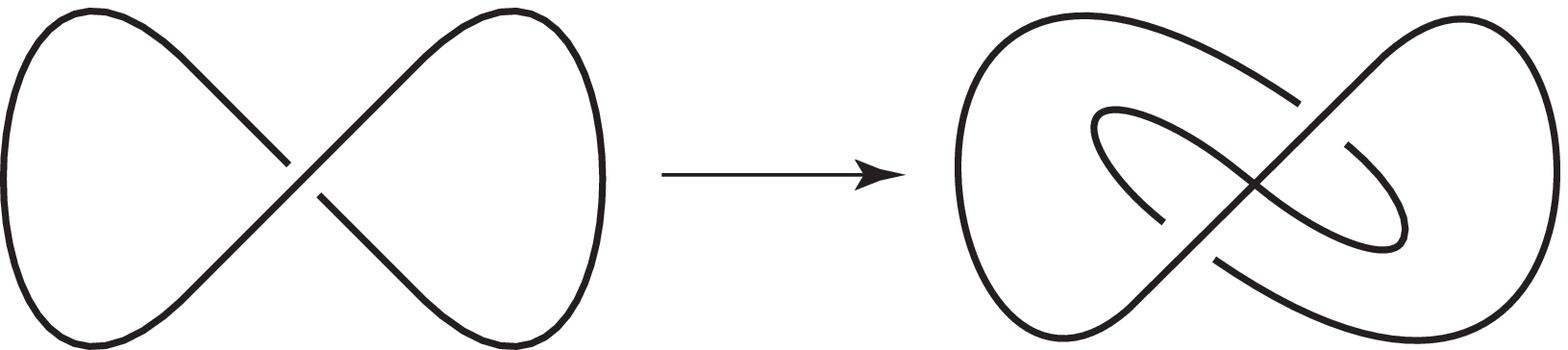}
      \put(30,-13){$D$}
      \put(127,-13){$D(2)$}
    \end{overpic}
  \end{center}
  \vspace{1em}
  \caption{$D(2)$ which is obtained from $D$ by the multiplexing of the crossing is not virtual isotopic to the trivial knot diagram.}
  \label{rem-virtual}
\end{figure}

%%%%%%%%%% Generalized link groups %%%%%%%%%%
\section{Generalized link groups}
Kelly~\cite{Kelly} and Wada~\cite{W}, independently, introduced a family of link invariants $G_{m}$ $(m\in\mathbb{Z})$ which are groups
generalizing the fundamental group of the complement of a classical link in the $3$-sphere $S^{3}$. 
Let $D$ be an oriented classical link diagram of a classical link $L$. 
The {\em generalized link group $G_{m}(D)$} of $D$ is defined as follows: 
Each arc of $D$ yields a generator, and each crossing of $D$ gives a relation 
as shown in Figure~\ref{wadarelation}. 
(Note that $G_{1}(D)\cong\pi_{1}(S^{3}\setminus L)$.) 
In~\cite{Kelly,W}, they proved that $G_{m}(D)$ is a classical link invariant. 
As we mentioned in Introduction, $G(D(m,\ldots,m))$ is isomorphic to $G_{m}(D)$. 
Hence, $D(m,\ldots,m)$ gives us a geometrical point of view for $G_{m}(D)$. 
Moreover, Theorem~\ref{th-multiplexing} implies that $G_{m}$ can be defined for not only classical link diagrams but also virtual link diagrams, 
and it is a welded link invariant.

\begin{figure}[htbp]
  \begin{center}
    \begin{overpic}[width=1.5cm]{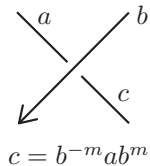}
      \put(8,38){$a$}
      \put(45,38){$b$}
      \put(38,9){$c$}
      \put(-3,-15){$c=b^{-m}ab^{m}$}
    \end{overpic}
  \end{center}
  \vspace{1em}
  \caption{A relation of the generalized link group $G_{m}(D)$}
  \label{wadarelation}
\end{figure}

It is well-known that 
the square knot $SK$ and the granny knot $GK$ are a pair of distinct knots with isomorphic fundamental groups. 
C. Tuffley~\cite{T} proved that $G_{m}(SK)$ and $G_{m}(GK)$ are not isomorphic for $m$ with $m\geq 2$. 
Moreover, 
S. Nelson and W.D. Neumann~\cite{NN} proved the following theorem.

\begin{theorem}\cite[Theorem 1.1]{NN}
\label{NN}
Let $m$ be an integer with $m\geq 2$,  
and let $D,D'$ be classical knot diagrams. 
$D$ is equivalent to $D'$ or $D'_{*}$ 
if and only if $G_{m}(D)\cong G_{m}(D')$, 
where $D'_{*}$ is the mirror image of $D'$. 
\end{theorem}

This theorem together with Theorem~\ref{th-multiplexing} implies the following.

\begin{theorem}
\label{th-iff}
Let $m$ be a non-zero integer $m$, 
and let $D,D'$ be classical knot diagrams. 
$D$ is equivalent to $D'$ or $D'_{*}$ 
if and only if $D(m)$ is welded isotopic to $D'(m)$ or $(D'(m))_{*}$.
\end{theorem}

\begin{proof}
Since we have that $D'_{*}(m)=(D'(m))_{*}$, 
the only if part immediately holds by Theorem~\ref{th-multiplexing}.

Thus, let us prove the if part. 
For $m=1$, it is trivial. 
Suppose that $m\geq 2$.  
If $D(m)$ is welded isotopic to $D'(m)$, 
then $G(D(m))\cong G(D'(m))$.
Therefore, $G_{m}(D)\cong G_{m}(D')$.
If $D(m)$ is welded isotopic to $(D'(m))_{*}=D'_{*}(m)$, 
then $G(D(m))\cong G(D'_{*}(m))$, and hence 
$G_{m}(D)\cong G_{m}(D'_{*})$. 
By Theorem~\ref{NN}, $D$ is equivalent to $D'$ or $D'_{*}$. 
If $m\leq -1$, then it is not hard to see that $D(|m|)$ and $(D(m))(-1)$ are welded isotopic. 
Hence, Theorem~\ref{th-multiplexing} implies that
if $D(m)$ and $D'(m)$ are welded isotopic, then 
$D(|m|)$ and $D'(|m|)$ are welded isotopic. 
Therefore, the proof follows from the case when $m\geq 1$.
\end{proof}

%%%%%%%%%% Proof of Theorem %%%%%%%%%%
\section{Proof of Theorem~\ref{th-multiplexing}}
In this section, we will give a proof of Theorem~\ref{th-multiplexing}.
Let us first prove the following lemma.

\begin{lemma}
\label{lem-moves}
The local moves A, B, $\text{C}^{+}$ and $\text{C}^{-}$ illustrated in Figure~\ref{moves} are realized by welded isotopy. 
Here, the square bounded by dashed lines in the move B may contain virtual crossings but not classical crossings. 
\end{lemma}

\begin{figure}[htbp]
  \begin{center}
    \begin{overpic}[width=12cm]{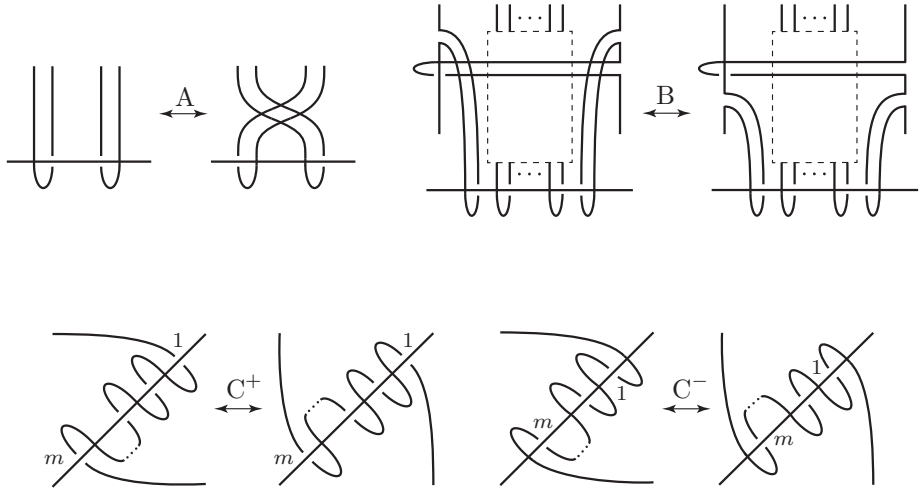}
    \put(63,144){A}
      \put(243,145){B}
      \put(82,34){$\text{C}^{+}$}
      \put(249,34){$\text{C}^{-}$}
      \put(62,53){{\footnotesize $1$}}
      \put(148,53){{\footnotesize $1$}}
      \put(14,9){{\footnotesize $m$}}
      \put(100,9){{\footnotesize $m$}}
      \put(228,33){{\footnotesize $1$}}
      \put(301,43){{\footnotesize $1$}}
      \put(197,22.5){{\footnotesize $m$}}
      \put(287,16){{\footnotesize $m$}}
    \end{overpic}
  \end{center}
  \caption{Local moves A, B, $\text{C}^{+}$ and $\text{C}^{-}$ realized by welded isotopy}
  \label{moves}
\end{figure}

\begin{proof}
\noindent
{\bf Move A}.~
See Figure~\ref{pf-tailexch}

\noindent
{\bf Move B}.~
See Figure~\ref{pf-sliding2}, where V denotes virtual isotopy.

\noindent
{\bf Moves $\text{C}^{+}$ and $\text{C}^{-}$}.~Let F be the local move illustrated in Figure~\ref{twist} which is realized by a detour move. 
Figure~\ref{pf-multiexch} (resp. Figure~\ref{pf-multiexch2}) indicates the proof for move $\text{C}^{+}$ (resp. $\text{C}^{-}$). 
While the proof is described only when $m=4$ in Figures~\ref{pf-multiexch} and~\ref{pf-multiexch2}, 
it is essentially same for any cases. 
\end{proof}

\begin{figure}[htbp]
  \begin{center}
    \begin{overpic}[width=12cm]{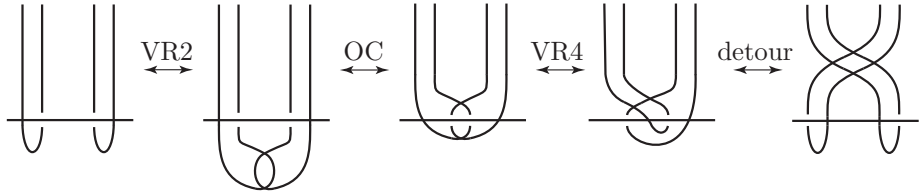}
      \put(50,49){VR2}
      \put(126,49){OC}
      \put(196,49){VR4}
      \put(266,49){detour}
    \end{overpic}
  \end{center}
  \caption{Proof for move A}
  \label{pf-tailexch}
\end{figure}

\begin{figure}[htbp]
  \begin{center}
    \begin{overpic}[width=12cm]{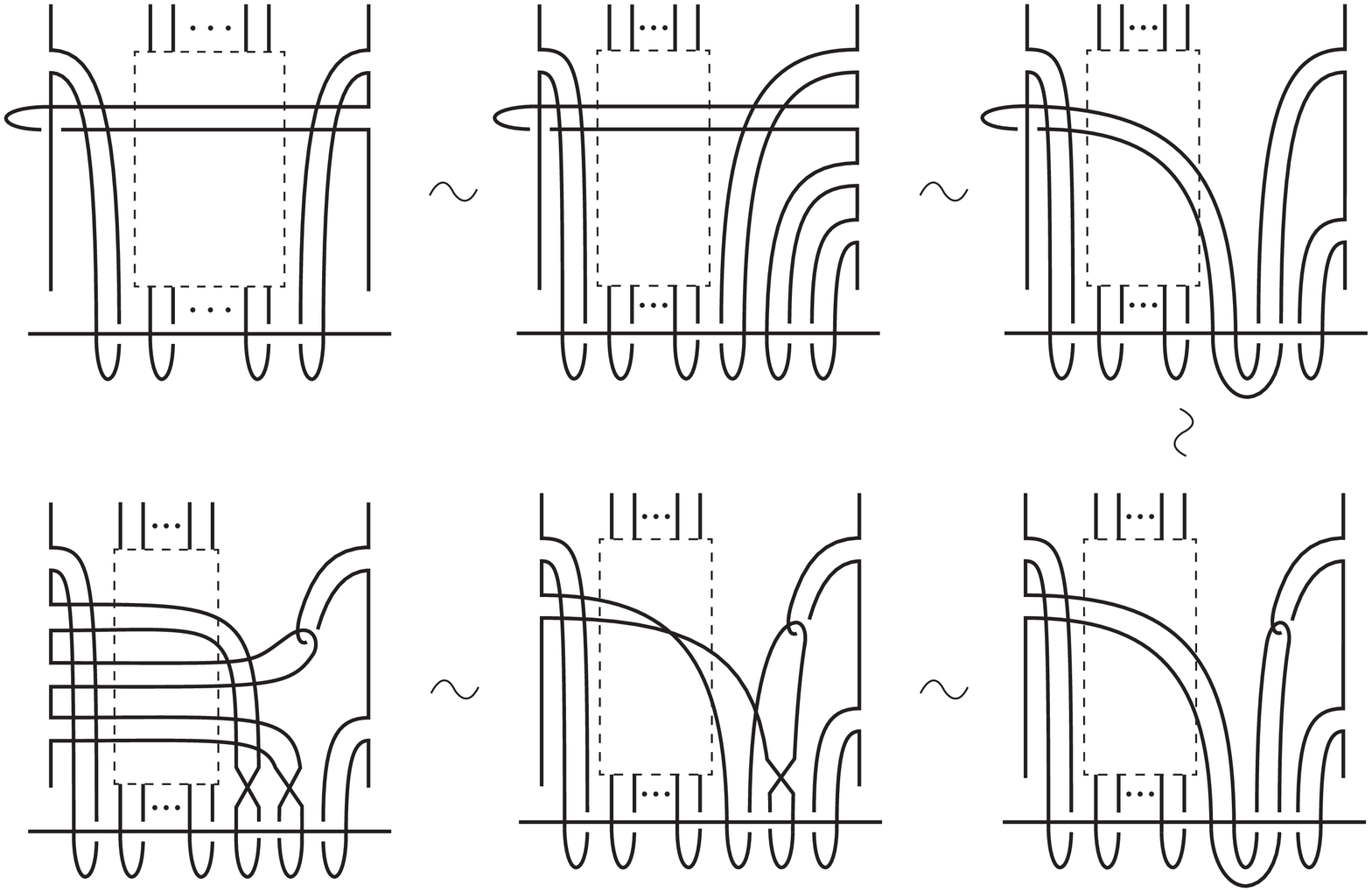}
      \put(110,178){V}
      \put(232,178){V}
      \put(301,110){V+OC}
      \put(110,54){V}
      \put(232,54){V}
    \end{overpic}
  \end{center}
  %\caption{}
  %\label{pf-sliding}
\end{figure}
\begin{figure}[htbp]
  \begin{center}
    \vspace{-1.3em}
    \begin{overpic}[width=12cm]{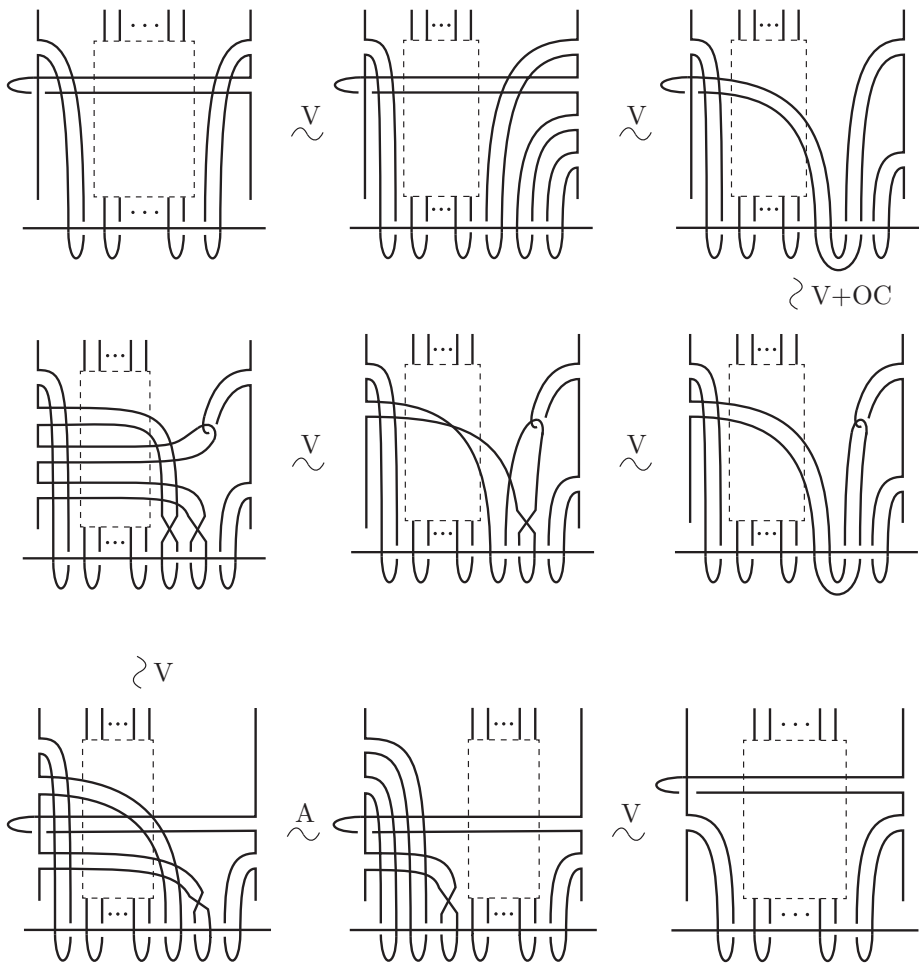}
      \put(55,106){V}
      \put(108,53){A}
      \put(230,53){V}
    \end{overpic}
  \end{center}
  \caption{Proof for move B}
  \label{pf-sliding2}
\end{figure}

\begin{figure}[htbp]
  \begin{center}
    \begin{overpic}[width=4.5cm]{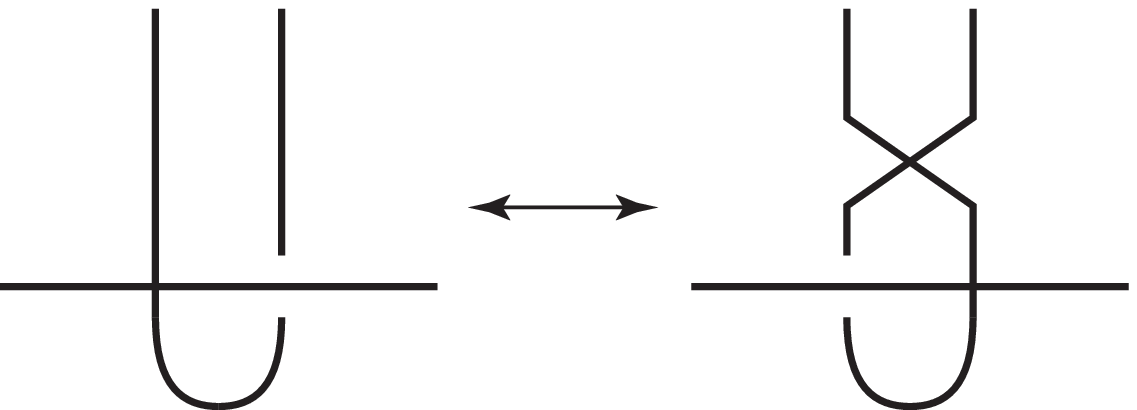}
      \put(61.5,26){F}
    \end{overpic}
  \end{center}
  \caption{Move F}
  \label{twist}
\end{figure}

\begin{figure}[htbp]
  \begin{center}
    \begin{overpic}[width=13cm]{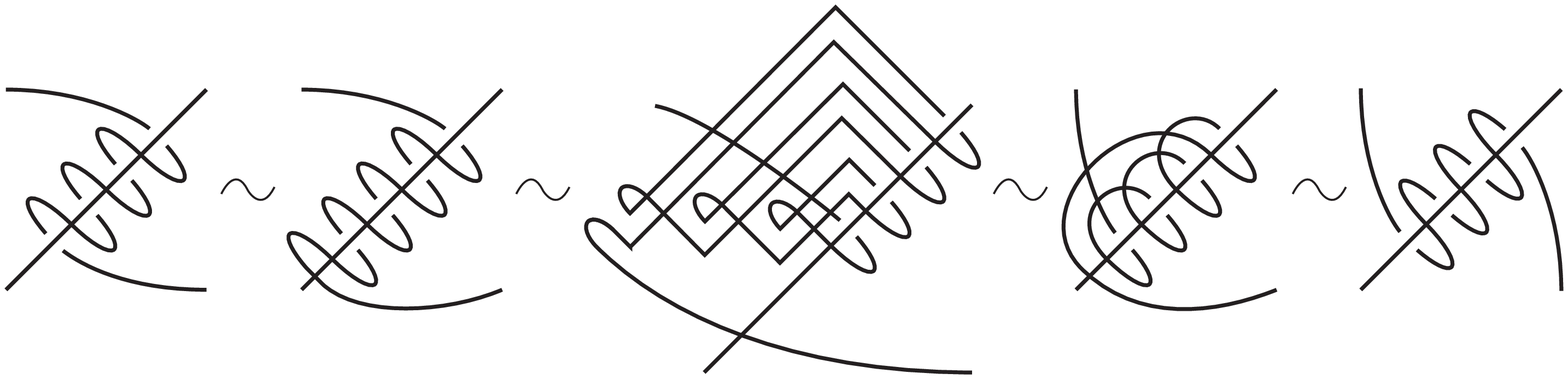}
      \put(55,48){V}
      \put(117,48){A$+$F}
      \put(237,48){V}
      \put(308,48){V}
    \end{overpic}
  \end{center}
  \caption{Proof for move $\text{C}^{+}$}
  \label{pf-multiexch}
\end{figure}

\begin{figure}[htbp]
  \begin{center}
    \begin{overpic}[width=10cm]{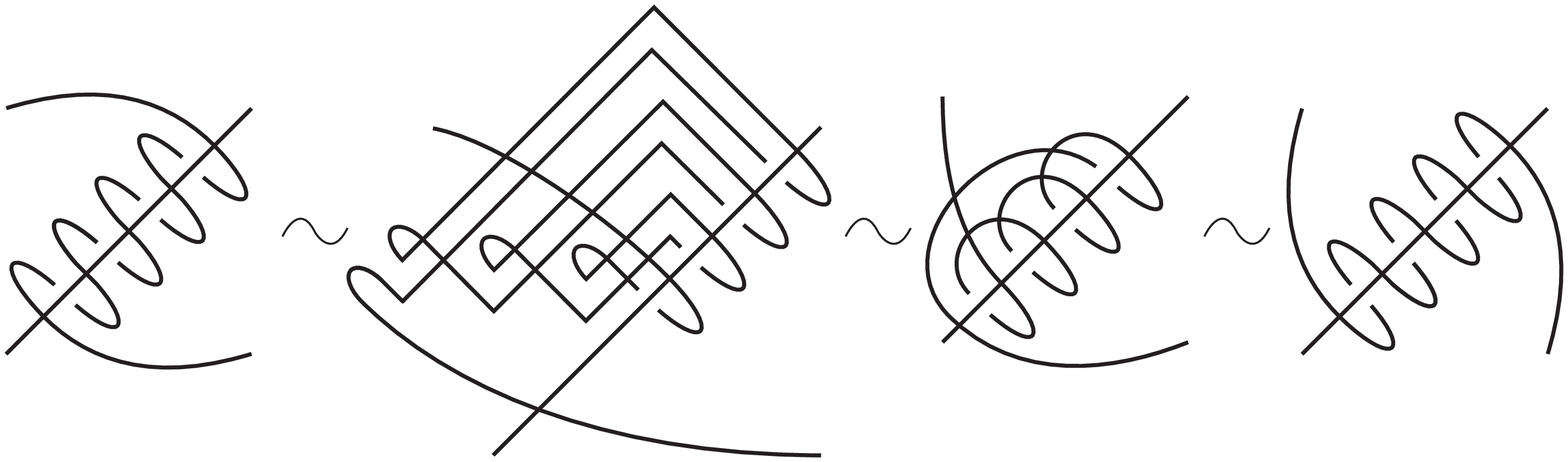}
      \put(46,45){A+F}
      \put(156,45){V}
      \put(221.5,45){V}
    \end{overpic}
  \end{center}
  \caption{Proof for move $\text{C}^{-}$}
  \label{pf-multiexch2}
\end{figure}

\begin{proof}[Proof of Theorem~\ref{th-multiplexing}]
It is enough to show that 
if $D$ and $D'$ are related by one of five moves R1, R2, R3, VR4 and OC, 
then $D(m_{1},\ldots,m_{n})$ and $D'(m_{1},\ldots,m_{n})$ are welded isotopic.

By using move $\text{C}^{+}$ or $\text{C}^{-}$, it is not hard to see that 
if $D$ and $D'$ are related by either R1 or R2, 
then $D(m_{1},\ldots,m_{n})$ and $D'(m_{1},\ldots,m_{n})$ are welded isotopic. 

If $D$ and $D'$ are related by a single VR4, 
then $D(m_{1},\ldots,m_{n})$ and $D'(m_{1},\ldots,m_{n})$ are related by a  detour move. 

If $D$ and $D'$ are related by a single R3, 
then $D(m_{1},\ldots,m_{n})$ and $D'(m_{1},\ldots,m_{n})$ are related by 
a finite sequence of virtual isotopy 
and moves~A, B, $\text{C}^{\pm}$ and F. 
Figure~\ref{pf-R3} indicates the proof when $m_{i}=3$ and $m_{j}=2$. 
In the general case, the proof is essentially same, 
where move $\text{C}^{-}$ is used instead of $\text{C}^{+}$ when $m_{i}$ is negative. 

If $D$ and $D'$ are related by a single OC, 
then by similar deformations as in Figure~\ref{pf-R3},
$D(m_{1},\ldots,m_{n})$ and $D'(m_{1},\ldots,m_{n})$ are related by 
a finite sequence of virtual isotopy and moves~A and $\text{C}^{\pm}$.
\end{proof}

\begin{figure}[htbp]
  \begin{center}
    \begin{overpic}[width=12cm]{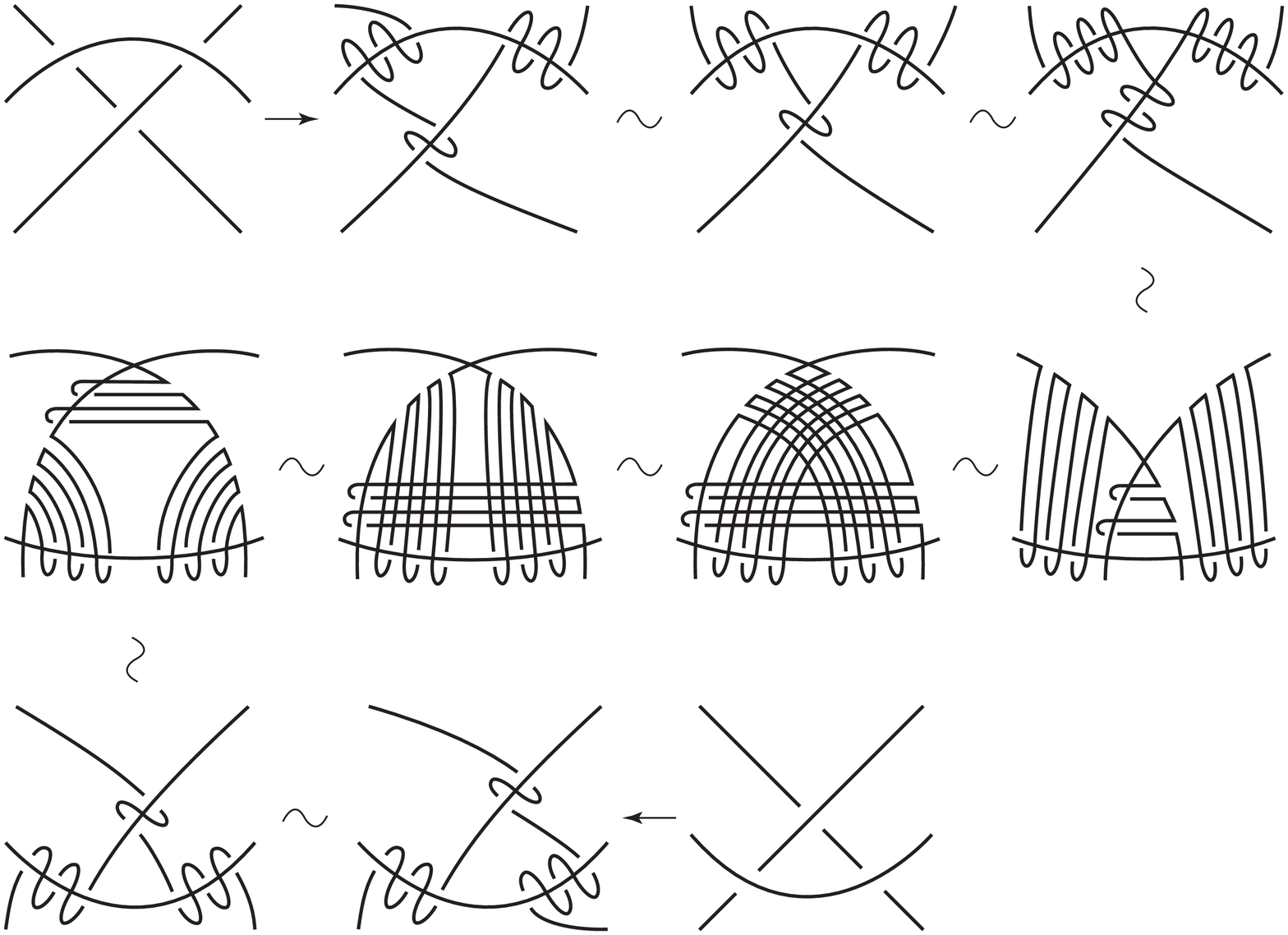}
      \put(30,174){$D$}
      \put(-20,228){$D_{i}\rightarrow$}
      \put(-13,200){$D_{j}\rightarrow$}
      \put(89,174){$D(m_{1},\ldots,m_{n})$}
      \put(166,221){$\text{C}^{+}$}
      \put(259,221){V}
      
      \put(310,167){V}
      \put(77,128){B}
      \put(159,128){A+F}
      \put(256,128){V}
      
      \put(41,68){V}
      \put(77,35){$\text{C}^{+}$}
      \put(93,-14){$D'(m_{1},\ldots,m_{n})$}
      \put(211,-14){$D'$}
    \end{overpic}
  \end{center}
  \vspace{1em}
  \caption{$D(m_{1},\ldots,m_{n})$ and $D'(m_{1},\ldots,m_{n})$ are related by a finite sequence of virtual isotopy and moves~A, B, $\text{C}^{+}$ and F when $m_{i}>0$.}
  \label{pf-R3}
\end{figure}

\begin{remark}
By using {\em Arrow calculus}, given by J.-B. Meilhan and the third author in~\cite{MY},  
we could prove Theorem~\ref{th-multiplexing} more simply.
It might be also possible to show Theorem~\ref{th-multiplexing} by using Gauss diagram. 
While our proof looks complicated, 
it is done by combining elementary deformations and, in particular, self-contained. 
\end{remark}

%%%%%%%%%% Examples %%%%%%%%%%
\section{Examples}
We are curious to have new {\em classical} link invariants from welded link invariants via the multiplexing of crossings. 
In fact, we have the following example.

\begin{example}
\label{ex-newinv}
Let $D=D_{1}\cup D_{2}\cup D_{3}$ be the ordered oriented $3$-component classical link diagram illustrated in Figure~\ref{ex-Alex}. 
Then, the Alexander polynomial $\Delta_{D}(t)$ of $D$ is $0$. 
On the other hand, $\Delta_{D(m_{1},m_{2},m_{3})}(t)=g(t)(t^{m_{1}}-t^{m_{2}})^{2}(1-t^{m_{3}})$, 
where $g(t)=\gcd{\{1-t^{m_{1}},1-t^{m_{2}},1-t^{m_{3}}\}}$. 
Therefore, $\Delta_{D(m_{1},m_{2},m_{3})}(t)$ is non-trivial for some $(m_{1},m_{2},m_{3})$
while $\Delta_{D}(t)$ vanishes. 
We remark that $D(m_{1},m_{2},m_{3})$ is not welded isotopic to a classical link diagram 
when $m_1\neq m_2$ 
since 
the intersection number of the 1st and 2nd components of $D(m_{1},m_{2},m_{3})$ is equal to $m_{1}-m_{2}$ $(\neq 0)$. 
\end{example}

\begin{figure}[htbp]
  \begin{center}
    \begin{overpic}[width=5cm]{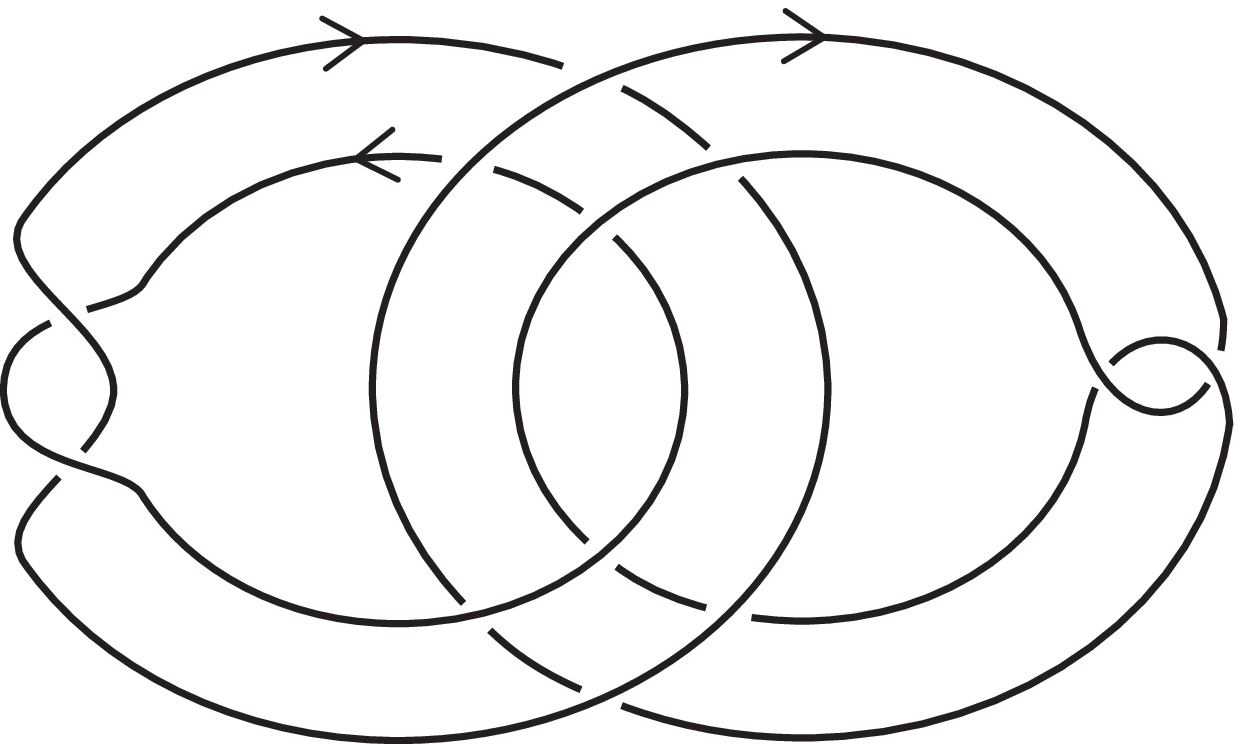}
      \put(20,83){$D_{1}$}
      \put(28,54){$D_{2}$}
      \put(106,83){$D_{3}$}
      \put(26,-15){$D=D_{1}\cup D_{2}\cup D_{3}$}
    \end{overpic}
  \end{center}
  \vspace{1em}
  \caption{An ordered oriented $3$-component classical link diagram with vanishing Alexander polynomial}
  \label{ex-Alex}
\end{figure}

In the example above, the $3$-variable Alexander polynomial of $D$ does not vanish. 
So far, we do not know if there is a classical link with vanishing multi-variable Alexander polynomial such that our invariants via the multiplexing of crossings survive. 
But, we have the following example.

\begin{example}
Let $D=D_{1}\cup D_{2}\cup D_{3}$ (resp. $D'=D'_{1}\cup D'_{2}\cup D'_{3}$) be the ordered oriented $3$-component virtual link diagram illustrated in the left-hand (resp. right-hand) side of Figure~\ref{ex-Alex2}. 
Then, the $3$-variable Alexander polynomials of $D$ and $D'$ 
are both equal to $(1-t_1)(1-t_2)(1-t_3)$. 
However, $\Delta_{D(m_{1},m_{2},m_{3})}(t)=(1-t^{m_{1}})^{2}(1-t^{m_{2}})(1-t^{m_{3}})$ and 
$\Delta_{D'(m_{1},m_{2},m_{3})}(t)=(1-t^{m_{1}})(1-t^{m_{2}})^{2}(1-t^{m_{3}})$. 
Therefore, $D$ and $D'$ can be distinguished by 
the 1-variable Alexander polynomials of 
$D(m_{1},m_{2},m_{3})$ and $D'(m_{1},m_{2},m_{3})$ while 
 the $3$-variable Alexander polynomials of $D$ and $D'$ coincide. 
\end{example}

\begin{figure}[htbp]
  \begin{center}
    \begin{overpic}[width=8cm]{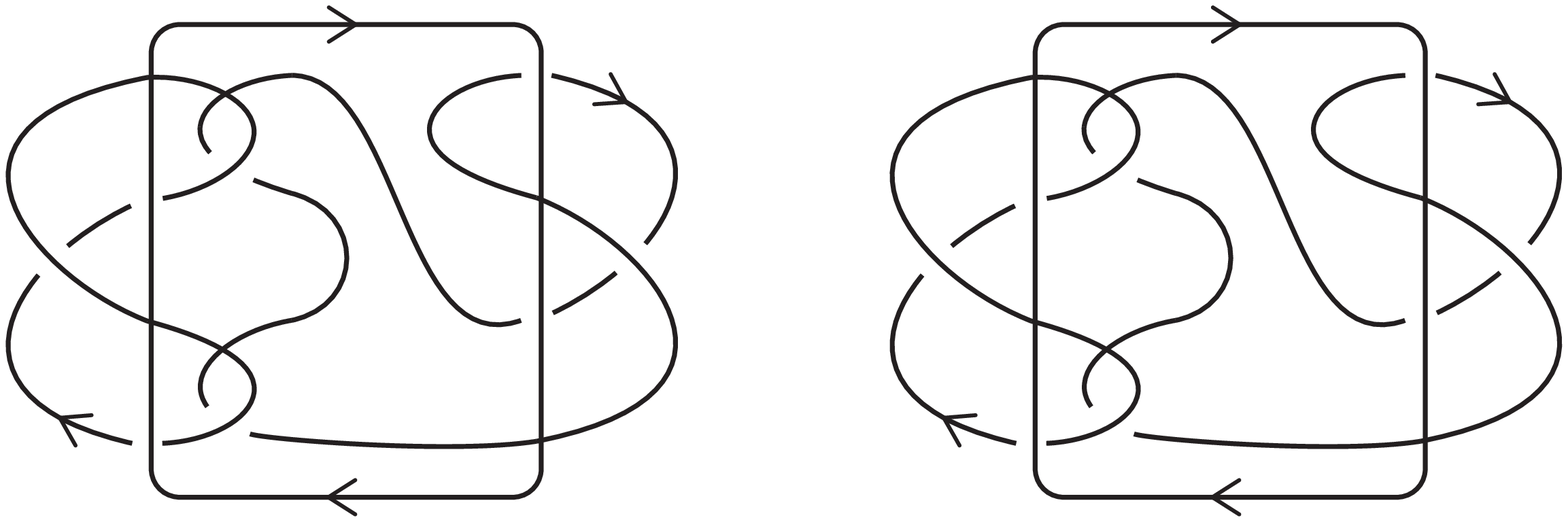}
      \put(-13,35){$D_{2}$}
      \put(44,79){$D_{1}$}
      \put(96,35){$D_{3}$}
      \put(9,-15){$D=D_{1}\cup D_{2}\cup D_{3}$}
      \put(118,35){$D'_{1}$}
      \put(173,79){$D'_{2}$}
      \put(227,35){$D'_{3}$}
      \put(137,-15){$D'=D'_{1}\cup D'_{2}\cup D'_{3}$}
    \end{overpic}
  \end{center}
  \vspace{1em}
  \caption{Two ordered oriented $3$-component virtual link diagrams with the same $3$-variable Alexander polynomial}
  \label{ex-Alex2}
\end{figure}

Suppose that each $m_{i}$ is equal to either $0$ or $1$. 
Then by the definition of the multiplexing of crossings, 
an invariant of $D(m_{1},\ldots,m_{n})$ might be weaker than that of $D$. 
(Note that $D(1,\ldots,1)=D$ and $D(0,\ldots,0)$ is a diagram of the $n$-component trivial link.) 
But even if some $m_{i}$'s are $0$, 
it seems still interesting to consider $D(m_{1},\ldots,m_{n})$. 
Because it would give us useful invariants that are handled easily. 
For example, we have the following.

\begin{example}
Let $D=D_{1}\cup D_{2}\cup D_{3}$ be the ordered oriented $3$-component link diagram illustrated in the left-hand side of Figure~\ref{ex-Borromean}. 
Then, the second Alexander polynomial of $D(1,1,0)$ is equal to $(1-t)^{2}$. 
Hence, $D(1,1,0)$ provides a concise way to 
determine that $D$ is non-trivial. 
\end{example}

\begin{figure}[htbp]
  \begin{center}
    \begin{overpic}[width=7cm]{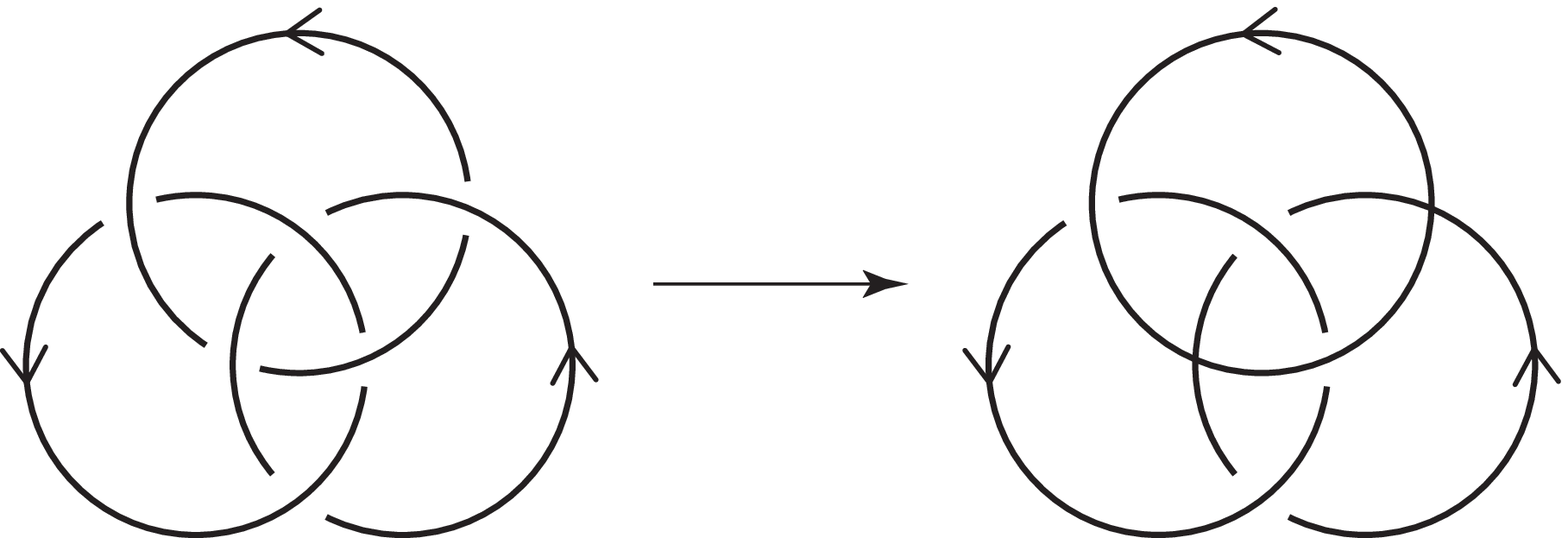}
      \put(13,64){$D_{1}$}
      \put(-7,7){$D_{2}$}
      \put(72,7){$D_{3}$}
      \put(-3,-15){$D=D_{1}\cup D_{2}\cup D_{3}$}
      \put(141,-15){$D(1,1,0)$}
    \end{overpic}
  \end{center}
  \vspace{1em}
  \caption{}
  \label{ex-Borromean}
\end{figure}

%%%%%%%%%% Acknowledgements %%%%%%%%%%
\begin{acknowledgements}
The authors would like to thank Professor J. Scott Carter for 
informing a result in~\cite{NN} 
which is helped showing Theorem~\ref{th-iff}.
This work was supported by JSPS KAKENHI Grant Numbers 
JP26400098, JP17J08186, JP17K05264.
\end{acknowledgements}

%%%%%%%%%% References %%%%%%%%%%

%%%%%%%%%%%%%%%%%%%%%%%%%%%%%%%%%%%%%%%%%%%%%%%%%%%%%%%
\end{document}